\documentclass[12pt]{amsart}

\title{Strongly graded groupoid and directed graph algebras}
\author{Lisa Orloff Clark, Ellis Dawson and Iain Raeburn}

\address{Department of Mathematics and Statistics, Victoria University of Wellington, PO Box 600, Wellington 6140, New Zealand}
\thanks{This research was supported by a Marsden grant from the Royal Society of New Zealand.}



\usepackage{amsmath,amsthm,amssymb, color,soul,enumerate,mathrsfs}
\usepackage{hyperref}
\usepackage{cleveref}
\usepackage[all]{xy}
\usepackage{graphicx}
\usepackage{mathdots}
\usepackage{MnSymbol}

\newtheorem{thm}{Theorem}
\newtheorem{prop}[thm]{Proposition}
\newtheorem{lem}[thm]{Lemma}
\newtheorem{cor}[thm]{Corollary}

\newtheorem*{thm*}{Theorem}
\theoremstyle{definition}

\newtheorem*{dfn*}{Definition}

\theoremstyle{remark}
\newtheorem{rmk}[thm]{Remark}

\newcommand{\C}{\mathbb{C}}
\newcommand{\N}{\mathbb{N}}

\newcommand{\T}{\mathbb{T}}
\newcommand{\Z}{\mathbb{Z}}

\def\LC{\ensuremath{L_\mathbb{C}(E)}}
\def\LK{\ensuremath{L_K(E)}}

\newcommand{\Aut}{\operatorname{Aut}}

\newcommand{\lsp}{\operatorname{span}}

\newcommand{\supp}{{\operatorname{supp}}}

\begin{document}

\begin{abstract}
We show the reduced $C^*$-algebra of a graded ample groupoid is a strongly graded $C^*$-algebra if and only if the corresponding Steinberg algebra is a strongly graded ring.  We apply this result to get a theorem about  the Leavitt path algebra and $C^*$-algebra of an arbitrary graph.  \end{abstract}

\maketitle

The prototypical example of a strongly $\Z$-graded ring is the ring $K[x,x^{-1}]$ of Laurent polynomials over a field $K$. The homogeneous components are the sets $Kx^p$, and for $m,n\in\Z$ such that $m+n=p$, every element $rx^p\in Kx^p$ can be realised as a product $(rx^m)(1x^n)$ of elements in the homogeneous components $Kx^m$ and $Kx^n$. The $C^*$-algebraic analogue of $K[x,x^{-1}]$ is the $C^*$-algebra $C(\T)$ of continuous functions on the unit circle $\T\subseteq\C$, which by the Stone--Weierstrass theorem contains the trigonometric polynomials as a dense subspace. This is the prototype of  a strongly $\mathbb{Z}$-graded $C^*$-algebra with homogeneous subspaces $C(\T)_n=\C z^n$.  

To be more precise, 
suppose that $A$ is a $C^*$-algebra and $\Gamma$ is a group. Following Exel \cite[Definition~16.2]{E}, we say that $A$ is  a \emph{$\Gamma$-graded  $C^*$-algebra} if there are linearly independent closed subspaces $\{A_{\gamma}:\gamma\in \Gamma\}$ of $A$ such that for every $\alpha,\beta\in \Gamma$, we have
\begin{enumerate}
\item$A_{\alpha}\cdot A_{\beta}:=\lsp\{ab:a\in A_{\alpha},b\in A_{\beta}\}\subseteq A_{\alpha\beta}$,
\item $A_{\alpha}^*:=\{a^*:a\in A_{\alpha}\}\subseteq A_{\alpha^{-1}}$ and
\item $\bigoplus_{\gamma\in \Gamma}A_{\gamma}$ is dense in $A$.
\end{enumerate} 
We say that $A$ is a \emph{strongly $\Gamma$-graded $C^*$-algebra} if in addition $A_{\alpha}\cdot A_{\beta}$ is dense in $A_{\alpha\beta}$ for all $\alpha\beta\in \Gamma$.

Our work is inspired by a recent result of Clark, Hazrat and Rigby in \cite{CHR} which identifies the groupoids for which the Steinberg algebra is a strongly graded ring.  Given an ample groupoid $G$ with Hausdorff unit space and a continuous cocycle $c:G \to \Gamma$,  
$c$ induces a grading on the Steinberg $R$-algebra $A_R(G)$.   The groupoid $G$ is strongly graded in the sense that $c^{-1}(\alpha)c^{-1}(\beta) = c^{-1}(\alpha\beta)$ for all $\alpha, \beta \in \Gamma$ if and only if $A_R(G)$ is a strongly $\Gamma$-graded ring \cite[Theorem~3.11]{CHR}.     
Strongly graded groupoids also appear in \cite{BM}.
In Proposition~\ref{prop:gd} we show that $c$ also induces a $\Gamma$-grading on the reduced $C^*$-algebra $C^*_r(G)$.  
We then show in Theorem~\ref{thm:gpd} that $C^*_r(G)$ is a strongly $\Gamma$-graded $C^*$-algebra if and only if the complex Steinberg algebra $A(G)$ is a strongly $\Gamma$-graded ring.   Our main tool in the proof is the injective, continuous linear map $j$ from $C^*_r(G)$ into the space of functions from $G$ to $\C$ that are bounded with respect to the uniform norm from \cite[Proposition~4.4.2]{renault}.

The key examples for us are the $C^*$-algebras $C^*(E)$ of directed graphs $E$, for which the gauge action $\gamma:\T\to \Aut C^*(E)$ gives a $\Z$-grading with
\[
C^*(E)_n:=\{a\in C^*(E):\gamma_z(a)=z^na\text{ for all }z\in \T\}.
\]
In \cite{CHR}, the authors identify the directed graphs $E$ for which the Leavitt path algebras $L_R(E)$ are strongly $\Z$-graded in \cite[Theorem~4.2]{CHR}.   (See also \cite[Theorem~1.3]{NO}.) 
We use our Theorem~\ref{thm:gpd} to show $\LC$ is strongly $\Z$-graded if and only if $C^*(E)$ is strongly $\Z$-graded. 
Thus the graph condition of \cite[Theorem~2.4]{CHR} (which we will describe soon) also characterises when $C^*(E)$ is strongly $\Z$-graded.  

Since we are proving a $C^*$-algebraic result, we use the convention that the spanning elements $\{\alpha\beta^*\}$ of $L_{\C}(E)$ are determined by pairs $\alpha,\beta\in E^*$ satisfying $s(\alpha)=s(\beta)$.  In other words, we use the conventions of \cite{R} rather than those of \cite{AAM}. With these conventions, a directed graph $E$ \emph{has property (Y)} if, for every infinite path $x\in E^\infty$ and every $k\in \N$, there are an initial segment $\alpha$ of $x$ and a finite path $\beta\in E^*$ such that $s(\alpha)=s(\beta)$ and $|\beta|-|\alpha|=k$, that is, $\beta$ has $k$ more edges than $\alpha$. The reason for the name is the picture:  

 
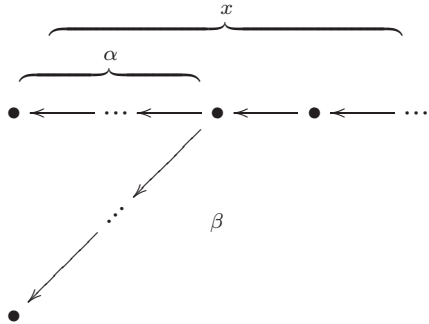
\[\xymatrix{
\bullet& \ar[l] \cdots\ar@{-}[r]^{\mbox{$\mkern60mu\overbrace{\phantom{\scriptscriptstyle abcdefghijklmnopqr}}^{\alpha}$}^{\mbox{$\mkern-100mu\overbrace{\phantom{\scriptscriptstyle abcdefghijklmnopqrstuvwxyz123456789}}^{x}$}}}  & \ar[l]\ar[dl]\bullet &  \ar[l]\bullet & \ar[l]\cdots  \\ & \udots \ar[dl]&_{\beta} &&  \\
 \bullet   &&&& &  & } 
 \]
 
We can state our main graph algebra result as follows.
\begin{thm}\label{mainthm}
Suppose that $E$ is a directed graph. Then the graph algebra $C^*(E)$ is strongly $\Z$-graded if and only if $E$ is row-finite, has no sources, and has property (Y).
\end{thm}

As a corollary, we get that $C^*(E)$ is a strongly $\Z$-graded $C^*$-algebra if and only if $L_{\C}(E)$ is strongly $\Z$-graded ring.  We conclude with a brief analysis of $(L(E))_0$, the core of a Leavitt path algebra which is dense inside of the core of $C^*(E)$.  We show that every element of the core of a Leavitt path algebra is contained in a finite-dimensional subalgebra.  Previously this was known for row-finite graphs, see for example the last line of the proof of \cite[Theorem~5.3]{AMP}. 

\section{Strongly graded Steinberg algebras and groupoid $C^*$-algebras}

We begin this section with some definitions.
Suppose $R$ is a ring.  
If $A$ and $B$ are subsets of $R$, then we write $AB$ for the set of all finite sums of elements of the form $ab$ where $a \in A$ and $b \in B$.   
Let $\Gamma$ be a group.   We say that $R$ is a \emph{ $\Gamma$-graded ring} if there are subgroups $\{R_{\gamma}:\gamma \in \Gamma\}$  in $R$ such that  for every $\alpha,\beta\in \Gamma$
\begin{enumerate}
\item\label{it1:grring} $R_{\alpha} R_{\beta} \subseteq R_{\alpha\beta}$ and  
\item\label{it2:grring} $\bigoplus_{\gamma \in \Gamma}R_{\gamma} =R.$
\end{enumerate}  
Further, we say $R$ is a \emph{strongly $\Gamma$-graded ring} if the containment in \eqref{it1:grring} above is in fact an equality.

We say a groupoid $G$ is ample if it is a topological groupoid and has a basis of compact open bisections.  
When $G$ is an ample groupoid such that $G^{(0)}$ is Hausdorff, the complex Steinberg algebra of $G$ is the space
\[A(G):= \lsp\{1_B : B \text{ is a compact open bisection}\}\]
where $1_X$ denotes the characteristic function from $G$ to $\C$ of $X$.
Addition and scalar multiplication are defined pointwise, and the convolution product is such that
\[1_B1_D = 1_{BD}\]
for compact open bisections $B$ and $D$.
If $c:G \to \Gamma$ is a continuous cocycle into a discrete group $\Gamma$, then the Steinberg algebra $A(G)$ is  a $\Gamma$-graded ring such that for $\gamma \in \Gamma$
\[A(G)_{\gamma}:=   \{f \in A(G) : \supp f \subseteq G_{\gamma}\}\]
where $G_{\gamma}:=c^{-1}(\gamma)$.  
Thus \[A(G) = \bigoplus_{\gamma \in \Gamma}A(G)_{\gamma}.\]
(See \cite[Lemma~3.11]{CS}.)
  
The next proposition shows that this structure also gives a grading on the groupoid $C^*$-algebra.
For this proposition (and also the following theorem) we use the injective continuous linear map 
\[j:C^*_r(G) \to \mathcal{B}(G), \]
where $ \mathcal{B}(G)$ denotes the normed vector space of bounded functions from $G$ to $\C$ with respect to $\| \cdot \|_{\infty}$.   This map restricts to the identity on $A(G)$.  (See, for example, \cite[page~3680]{CEPSS}.) 

\begin{prop}\label{prop:gd}
Let $G$ be an ample groupoid such that $G^{(0)}$ is Hausdorff and let $c:G \to \Gamma$ be a continuous cocycle into a discrete group $\Gamma$.  Then $C^*_r(G)$ is a $\Gamma$-graded $C^*$-algebra such that for each $\gamma \in \Gamma$ we have
\[C^*_r(G)_{\gamma} := \overline{A(G)_{\gamma}}.\]

\end{prop}

\begin{proof}
Fix $\alpha, \beta \in \Gamma$.  
Because   
$A(G)_{\alpha}\cdot A(G)_{\beta}$ is contained in the closed set $C^*_r(G)_{\alpha\beta}$,  
\[C^*_r(G)_{\alpha}\cdot C^*_r(G)_{\beta} \subseteq C^*_r(G)_{\alpha\beta}\] by the continuity of multiplication.
Similarly \[C^*_r(G)_{\gamma}^* \subseteq C^*_r(G)_{\gamma^{-1}}\]
 by the continuity of the involution $*$.
Next we verify that each $C^*_r(G)_{\gamma}$ is a linearly independent subspace. 
To do this, we claim that 
\[j(C^*_r(G)_{\gamma}) \subseteq \mathcal{B}(G_{\gamma}).\]   We have that $j$ is the identity map on $A(G)_{\gamma} \subseteq \mathcal{B}(G_{\gamma})$.   It is straightforward to check that $\mathcal{B}(G_{\gamma})$ is closed in $\mathcal{B}(G)$ with respect to $\| \cdot \|_{\infty}$ and hence the claim follows from the continuity $j$.

Now consider a finite linear combination $\sum a_i = 0$ where each $a_i \in C^*_r(G)_{\gamma_i}$ and the $\gamma_i$ are distinct elements of $\Gamma$. By linearity 
\[0=j(\sum a_i) = \sum j(a_i).\]  
Since the supports of $j(a_i)$ are all disjoint, $j(a_i)=0$ for every $i$.
Thus $a_i = 0$ because $j$ is an injective linear map.

Finally notice
\[A(G) = \bigoplus_{\gamma \in \Gamma}A(G)_{\gamma} \subseteq \bigoplus_{\gamma \in \Gamma}C^*_r(G)_{\gamma}
\subseteq C^*_r(G).\] Thus since $A(G)$ is dense in $C^*_r(G)$ by \cite[Proposition~6.7]{Steinberg}, we have that 
\[\bigoplus_{\gamma \in \Gamma}C^*_r(G)_{\gamma} \text{ is dense in }C^*_r(G).\]
\end{proof}

\begin{thm}
\label{thm:gpd}
Let $G$ be an ample groupoid such that $G^{(0)}$ is Hausdorff and let $c:G \to \Gamma$ be a continuous cocycle into a discrete group $\Gamma$.  The following are equivalent:
\begin{enumerate}
\item\label{it1:gpd} $G$ is strongly $\Gamma$-graded groupoid in the sense that $G_{\alpha}G_{\beta}=G_{\alpha\beta}$ for all $\alpha, \beta \in \Gamma$;
\item\label{it2:gpd} $A(G)$ is a strongly $\Gamma$-graded ring;
\item\label{it3:gpd} $C^*_r(G)$ is a strongly $\Gamma$-graded $C^*$-algebra.
\end{enumerate}
\end{thm}

\begin{rmk}
Using \cite[Theorem~3.11]{CHR}, we could replace item~\eqref{it2:gpd} in Theorem~\ref{thm:gpd} with the equivalent statement:  $A_R(G)$ is a strongly $\Gamma$-graded ring
for any commutative ring $R$ with identity.    
\end{rmk}

The equivalence of \eqref{it1:gpd} and \eqref{it2:gpd} is established in \cite[Theorem~3.11]{CHR}.  We establish the equivalence of \eqref{it2:gpd} and \eqref{it3:gpd}.
That \eqref{it3:gpd} implies \eqref{it2:gpd} is straightforward:  Take $\alpha, \beta \in \Gamma$.  Since $C^*_r(G)_{\alpha\beta}$ is closed and $C^*_r(G)$ is $\Gamma$-graded, 
\[C^*_r(G)_{\alpha}\cdot C^*_r(G)_{\beta}\]
 trivially has closure in  $C^*_r(G)_{\alpha\beta}$.
To see that this closure is all of  $C^*_r(G)_{\alpha\beta}$ , take $a \in  C^*_r(G)_{\alpha\beta}$.  Since 
 $C^*_r(G)_{\alpha\beta} = \overline{A(G)_{\alpha\beta}}$, there exist $f_n \in A(G)_{\alpha\beta}$ such that 
$f_n \to a$ in $C^*_r(G)$.  Since $A(G)$ is strongly graded, we have 
\[f_n \in A(G)_{\alpha}\cdot A(G)_{\beta} \subseteq C^*_r(G)_{\alpha}\cdot C^*_r(G)_{\beta}\]
for each $n$ 
and hence the limit $a$ belongs to 
\[\overline{ C^*_r(G)_{\alpha}\cdot C^*_r(G)_{\beta}}.\]

For the reverse implication, we use the following lemma.

\begin{lem}
\label{lem:agag}
Let $G$ be an ample groupoid such that $G^{(0)}$ is Hausdorff and let $c:G \to \Gamma$ be a continuous cocycle into a discrete group $\Gamma$.  Then for $\alpha,\beta \in \Gamma$ we have 
\[A(G)_{ \alpha}\cdot A(G)_{\beta} \subseteq \{f \in A(G) :  f(x) \neq 0 \implies x \in G_{\alpha}G_{\beta}\}.\]
\end{lem}

\begin{proof}
Observe that 
$A(G)_{ \alpha}\cdot A(G)_{\beta}$ is equal to
\[\label{eq:agag}
 \lsp\{1_{BD}: B \subseteq G_{\alpha} \text{ and } 
D \subseteq G_{\beta} \text{ are compact open bisections}\} 
\]
(see \cite[bottom of page 53]{CHR}).  
Since the support of a sum of functions is contained in the union of the supports of the summands, the containment follows.  
\end{proof}

\begin{proof}[Proof of Theorem~\ref{thm:gpd}]
To complete the proof, suppose $A(G)$ is not strongly graded.  Then $G$ is not strongly graded by \cite[Theorem~3.11]{CHR}, that is, by \eqref{it1:gpd} $\implies$ \eqref{it2:gpd} of Theorem~\ref{thm:gpd}.
  Then there exists $\gamma \in \Gamma$ and $x \in G_e \setminus G_{\gamma}G_{\gamma^{-1}}$
  by \cite[Lemma~3.1]{CHR}. Since $c$ is continuous, $G_e$ is open so we can find $B \subseteq G_e$ a compact open bisection containing $x$.  Since $1_B(x)=1$,  \[1_B \in A(G)_e \setminus A(G)_{\gamma}A(G)_{\gamma^{-1}}\] by Lemma~\ref{lem:agag}.  
 We show 
 \[1_B \notin \overline{A(G)_{\gamma}A(G)_{\gamma^{-1}}}\] 
using an adaptation of the argument in \cite[Proposition~3.3(i)]{CF}.   
By way of contradiction, suppose there exists a sequence \[(f_n) \subseteq A(G)_{\gamma} \cdot A(G)_{\gamma^{-1}}\] 
that converges to $1_B$ in $C^*_r(G)$.   Again we apply the injective continuous linear map $j:C^*_r(G) \to \mathcal{B}(G)$.   Since $j$ restricts to the identity map on $A(G)$ we have $j(f_n)=f_n \to j(1_B)=1_B$ in $\| \cdot \|_{\infty}$.  Recall that $1_B(x) \neq 0$.  However, $x \notin G_{\gamma}G_{\gamma^{-1}}$ so 
$f_n(x) = 0$ for all $n \in \mathbb{N}$ by Lemma~\ref{lem:agag}, which is a contradiction.  Therefore 
 \[f \notin \overline{A(G)_{\gamma}A(G)_{\gamma^{-1}}}\]  and $C^*_r(G)$ is not a strongly $\Gamma$-graded $C^*$-algebra.
\end{proof}

\section{Leavitt path algebras}

We want to connect properties of $C^*(E)$ and properties of $L_\C(E)$, and hence we want to able to view $L_{\C}(E)$ as a subalgebra of $C^*(E)$ using the homomorphism $\iota$ which takes a spanning element $\alpha\beta^*$ to the element $s_{\alpha}s_{\beta}^*$ of $C^*(E)$. That $\iota$ is an injection was established (for example) in \cite[\S1.3]{APS}, but with an annoying hypothesis of ``no sources''. So we pause to prove the following more general result.

\begin{prop}
Suppose that $E$ is a directed graph. Then the homomorphism $\iota:L_{\C}(E)\to C^*(E)$ is an isomorphism of $L_{\C}(E)$ onto the $*$-subalgebra 
\[
A:=\lsp\big\{s_{\alpha}s_{\beta}^*:\alpha,\beta\in E^*\text{ and }s(\alpha)=s(\beta)\big\}\quad\text{of } C^*(E).
\]
\end{prop}

\begin{proof}
We want to apply the Graded Uniqueness Theorem for Leavitt path algebras \cite[Theorem~2.2.15]{AAM}, but $C^*(E)$ is not $\Z$-graded in the algebraic sense, and hence that result does not apply to $\iota:L_{\C}(E)\to C^*(E)$. However, the $*$-subalgebra $A$ is $\Z$-graded, with 
\[
A_n=\lsp\big\{s_{\alpha}s_{\beta}^*:s(\alpha)=s(\beta)\text{ and }|\alpha|-|\beta|=n\big\}.
\]
Since we can always find Cuntz-Krieger $E$-families $\{S,P\}$ with $P_v\not=0$ for every $v\in E^0$ (see the top of \cite[page~8]{R}), every $p_v$ is non-zero. Thus \cite[Theorem~2.2.15]{AAM} implies that $\iota$ is an injection of $L_{\C}(E)$ into $A$; since every $s_\alpha s_{\beta}^*$ belongs to the range of $\iota$, it is an isomorphism, as claimed.
\end{proof}

 Next we relate the grading spaces $C^*(E)_n$ for the $C^*$-algebra to those of the Leavitt path algebra. We define $\Phi_n:C^*(E)\to C^*(E)$ in terms of the gauge action $\gamma:\T\to \Aut C^*(E)$ by 
\[
\Phi_n(a)=\int_{\T}w^{-n}\gamma_w(a)\,dw.
\]
Left invariance of the Haar integral on $\T$ implies that $\Phi_n$ is a norm-decreasing $C^*(E)^\gamma$-bilinear map of $C^*(E)$ onto $C^*(E)_n$ such that
\[
\Phi_n(s_\mu s_\nu^*)=\begin{cases}s_\mu s_\nu^*&\text{if $|\mu|-|\nu|=n$, and}\\
0&\text{otherwise.}
\end{cases}
\]
The map $\Phi_n\circ\iota$ is a bounded linear map of $\LC$ onto a dense subspace of $C^*(E)_n$, and its restriction to $\LC_n$ is an injection of $\LC_n$ onto this dense subspace of $C^*(E)_n$. Hence, modulo the canonical injection $\iota$ of $\LC$ in $C^*(E)$, $C^*(E)_n$ is the closure of $\LC_n$. We suppress $\iota$, and write $C^*(E)_n=\overline{\LC_n}$.

\begin{prop}\label{C*-ring}
Suppose that $E$ is a directed graph. Then $C^*(E)$ is a strongly $\Z$-graded $C^*$-algebra if and only if $\LC$ is a strongly $\Z$-graded ring. 
\end{prop}

With the following lemma, the proof of Proposition~\ref{C*-ring} is an immediate corollary to Theorem~\ref{thm:gpd}.

\begin{lem}
\label{lem:help}
For any directed graph $E$, there is an ample Hausdorff amenable groupoid $G_E$ and an isomorphism $\pi:C^*(E) \to C^*(G_E)$ such that $\pi(C^*(E)_n) = C^*(G_E)_n$,  and such that 
the restriction of $\pi$ to $L_{\C}(E)$ is a $\Z$-graded (ring) isomorphism of $L_{\C}(E)$ onto $A(G_E)$.
\end{lem}

\begin{rmk}
Because $G_E$ is amenable, the full and reduced $C^*$-algebras are equal so we drop the \emph{r} subscript from the $C^*_r(G_E)$ from now on.
\end{rmk}

\begin{proof}
Let $G_E$ be the \emph{boundary path groupoid of $E$} as defined by Paterson on page 653 of \cite{P}\footnote{Paterson calls $G_E$ the \emph{path groupoid}}.  Then $G_E$ is ample and Hausdorff; see \cite[Theorem~2.4]{Rigby} for a nice proof of this.  The groupoid $G_E$ is amenable by \cite[Theorem~4.2]{P}.  We view $C^*(G_E)$ as the universal groupoid $C^*$-algebra.  In this sense, there is a representation $\pi_{max}:C_c(G_E) \to C^*(G_E)$ that is universal for representations of $C_c(G_E)$ and $\pi_{max}(C_c(G_E)$ is dense in $C^*(G_E)$.  
See \cite[Theorem~3.2.2]{Aidan} for more details.  
Since we view $A(G_E) \subseteq C_c(G_E)$ as subsets of $C^*(G_E)$, we identify $C_c(G_E)$ with $\pi_{max}(C_c(G_E)$ inside $C^*(G_E)$,which we can do because $\pi_{max}$ is injective by \cite[Corollary~3.3.4]{Aidan}.

One can show that the collection 
\[P_v:= 1_{Z(v)}, \quad S_e:= 1_{Z(e,s(e)}\]
for $v \in E^0$ and $e \in E^1$ is a Cuntz-Krieger $E$-family in $C^*(G_E)$.  The universal property of $C^*(E)$ gives a homomorphism $\pi:C^*(E) \to C^*(G_E)$ such that 
\[ \pi(p_v) = P_v \quad \text{for $v \in E^0$ and} \quad \pi(s_e)= S_e \text{ for $ e \in E^1 $}. \] 
(See \cite[page~42]{R}.)  Notice that the restriction of $\pi$ to $L_{\C}(E)$ is precisely the $\Z$-graded ring isomorphism from $L_{\C}(E)$  onto $A(G_E)$ in \cite[Example~3.2]{CS},  and hence  $\pi(\LC_n) = A(G_E)_n$.

Paterson shows in in \cite[Theorem~3.8]{P} and \cite[Theorem~3.1]{Pbook} that representations of $C^*(E)$ are in bijective correspondence with representations of $C_c(G_E)$ .  
So the universal representation $p_v, s_e$ of $C^*(E)$ corresponds to a representation
$\phi:C_c(G_E) \to C^*(E)$: indeed, by looking at the details of the proof of \cite[Theorem~3.8]{P}, we see that $\phi(P_v) = p_v$ and $\phi(S_e) = s_e$ for $v \in E^0$ and $e \in E^1$.
The universal property of $C^*(G_E)$ \cite[Theorem~3.2.2]{Aidan}  gives a homomorphism 
$\psi:C^*(G_E) \to C^*(E)$ such that $\psi = \phi$ on $C_c(G_E)$.
  
Notice that $\pi \circ \phi$ and $\phi \circ \pi$ restrict to identity maps on $A(G_E)$ and $\LC$ respectively.  We show that $\pi$ is a bijection by showing that $\pi \circ \psi$ and $\psi \circ \pi$ are the identity maps on $C^*(G_E)$ and $C^*(E)$ respectively.
Fix $a \in C^*(G_e)$.  Since $A(G_E)$ is dense in $C^*(G_E)$ by \cite[Proposition~6.7]{Steinberg},  $a = \lim f_n$ where $f_n \in A(G_E)$.  Then using that homomorphisms betweeen $C^*$-algebras are automatically continuous, we have
\[\pi(\psi(a))= \pi(\psi(\lim f_n)) = \lim \pi(\psi(f_n)) =\lim \pi(\phi(f_n)) = \lim f_n =a.\]
For the other direction, fix $b \in C^*(E)$.  Then $b = \lim b_n$ where $b_n \in \LC$, and
recalling that $\pi$ takes elements of $\LC$ to elements of $A(G_E)$, we deduce that 
\[\psi(\pi(b))= \psi(\pi(\lim(b_n))) = \lim(\psi(\pi(b_n))) = \lim(\phi(\pi(b_n))) =  \lim b_n = b.\]
Thus $\pi$ is an isomorphism with inverse $\psi$.    
Continuity of $\pi$ and $\psi$  imply $\pi(C^*(E)_n) = C^*_r(G_E)_n$,  and this completes the proof.
\end{proof}
 
\begin{proof}[Proof of Proposition~\ref{C*-ring}]
Because the restricted map $\pi$ from Lemma~\ref{lem:help} is a graded ring isomorphism,  $\LC$ is strongly $\Z$-graded if and only if $A(G_E)$ is strongly $\Z$-graded.  Now Theorem~\ref{thm:gpd} says this is equivalent to  $C^*(G_E)$ being  strongly $\Z$-graded.    Since $\pi$  preserves the $C^*$-grading,  the result follows.
\end{proof}

Now we have the peices needed to prove Theorem~\ref{mainthm}.  Recall that \cite[Theorem~4.2]{CHR} says $\LC$ is strongly $\Z$-graded if and only if $E$ is row-finite, has no sources and has property Y.  Thus Theorem~\ref{mainthm} follows from Proposition~\ref{C*-ring} above and \cite[Theorem~4.2]{CHR} 

In strongly graded algebras,  a lot of the structure of the algebra can be seen in the core, that is, in the homogeneous component of the identity $e$.  Thus we conclude by establishing that the core of a Leavitt path algebra is build from finite dimensional subalgebras in the following way. 

\begin{prop}\label{aisinfinitesub}
Suppose that $E$ is a directed graph and $K$ is a field. Then every $a\in \LK_0$ belongs to a finite-dimensional $*$-subalgebra of $\LK_0$.
\end{prop}

To prove this result, we need to understand the finite-dimensional subalgebras of $\LK_0$.
We start by choosing an enumeration $\{e_j:j\in\N\}$ of the countable set $E^1$. For $J\in \N$, we write 
\[
E^k_J:=\big\{\alpha\in E^k:\alpha_i\in\{e_j:1\leq j\leq J\}\text{ for all $i$}\big\}, 
\]
and $E^0_J:=\{s(e_j):1\leq j\leq J\}$. Notice that the sets $E^k_J$ are all finite, and $E^k_J\subseteq E^k_{J+1}$. We define 
\[
G_{k,J}:=\lsp\{\alpha\beta^*:\alpha,\beta\in E^k_J\}.
\] 
For $\beta,\gamma\in E^k_J$ we have $|\beta|=|\gamma|=k$, and hence we have
\[
\beta^*\gamma=\begin{cases} s(\beta)=s(\gamma)&\text{if }\beta=\gamma\\0&\text{otherwise.}\end{cases}
\]
It follows that
\[
G_{k,J}(v):=\lsp\{\alpha\beta^*:=\alpha,\beta\in E^k_Jv\}
\]
is a (finite-dimensional) matrix algebra, and 
\[
v\not= w\Longrightarrow G_{k,J}(v) G_{k,J}(w)=\{0\}.
\]
Thus 
\[
G_{k,J}=\bigoplus\limits_{v\in E^0_J}G_{k,J}(v)
\] 
is a finite direct sum of matrix algebras.

\begin{lem}\label{gkgl=gk}
Suppose that $k,l\in \N$ and $k>l$. Then for every $J\in \N$ and $v,w\in E^0_J$, we have
\[
G_{k,J}(v)G_{l,J}(w)\subseteq G_{k,J}(v).
\]
\end{lem}

\begin{proof}
For $\alpha\beta^*\in G_{k,J}$ and $\gamma\delta^*\in G_{l,J}(w)$, we have $|\beta|=k>l=|\gamma|$, and hence
\[
(\alpha\beta^*)(\gamma\delta^*)=\begin{cases} \alpha(\delta\beta')^*&\text{if }\beta=\gamma\beta'\\0&\text{otherwise.}\end{cases}
\]
Since $s(\delta\beta')=s(\beta)=s(\alpha)=v$, we deduce that $(\alpha\beta^*)(\gamma\delta^*)\in G_{k,J}(v)$ (it could be $0$, but that is in $G_{k,J}(v)$ too). A similar argument gives $(\gamma\delta^*)(\alpha\beta^*)\in G_{k,J}(v)$. 
\end{proof}

\begin{cor}\label{exfdalg}
For each $J\in \N$, the set
\[
F_{k,J}:=\lsp\Big(\bigcup\{G_{l,J}:l\leq k\}\Big)
\]
is a finite-dimensional $*$-subalgebra of $\LK_0$.
\end{cor}

\begin{proof}
The set $F_{k,J}$ is a vector subspace of $\LK_0$ and is closed under conjugation because each $G_{l,J}$ is. The lemma implies that it is closed under multiplication. It is finite-dimensional because there are only finitely many $G_{l,J}$ in play, and their union spans $F_{k,J}$.
\end{proof}

\begin{proof}[Proof of \Cref{aisinfinitesub}]
We suppose that $a\in \LK_0$. Then there are a finite set $S$ of pairs $(\alpha,\beta)\in E^*\times E^*$ such that $s(\alpha)=s(\beta)$ and $|\alpha|=|\beta|$, and scalars $c_{\alpha,\beta}$ such that
\[
a=\sum_{(\alpha,\beta)\in S}c_{\alpha,\beta}\alpha\beta^*.
\]
Since the set $S$ is finite, the set of vertices
\[
V:=\big\{v\in E^0:v=s(\alpha_i)\text{ or }s(\beta_i) \text{ for some $(\alpha,\beta)\in S$ and $i\leq |\alpha|$}\big\}
\]
is finite. Thus there exists $J$ such that $V\subseteq E^0_J$. Then with $k:=\max\{|\alpha|:(\alpha,\beta)\in S\}$, we have
\[
\alpha\beta^*\in G_{|\alpha|,J}\subseteq F_{k,J}\quad\text{for all $(\alpha,\beta)\in S$.}
\]
Thus $a=\sum c_{\alpha,\beta}\alpha\beta^*$ belongs to the subspace $F_{k,J}$, which by Corollary~\ref{exfdalg} is a finite-dimensional $*$-subalgebra of $\LK_0$.
\end{proof}

\end{document}